\documentclass[review]{elsarticle}

\usepackage{lineno,hyperref}
\modulolinenumbers[5]

\usepackage{amsmath}
\usepackage{amsfonts}
\usepackage{amsthm}
\usepackage{amssymb}
\usepackage{color}
\usepackage{graphicx}
\usepackage{dsfont}

\newtheorem{definition}{Definition}
\newtheorem{proposition}{Proposition}
\newtheorem{lemma}{Lemma}

\makeatletter
\def\ps@pprintTitle{%
	\let\@oddhead\@empty
	\let\@evenhead\@empty
	\def\@oddfoot{}%
	\let\@evenfoot\@oddfoot}
\makeatother

\begin{document}

\begin{frontmatter}

\title{On characterizations of the covariance matrix}

%% Group authors per affiliation:
\author{Joni Virta}
\address{Aalto University,\\ Department of Mathematics and Systems Analysis,\\ Finland\\ \medskip
University of Turku, \\ Department of Mathematics and Statistics,\\ Finland\\}

\begin{abstract}
The covariance matrix is well-known for its following properties: affine equivariance, additivity, independence property and full affine equivariance. Generalizing the first one leads into the study of scatter functionals, commonly used as plug-in estimators to replace the covariance matrix in robust statistics. However, if the application requires also some of the other properties of the covariance matrix listed earlier, the success of the plug-in depends on whether the candidate scatter functional possesses these. In this short note we show that under natural regularity conditions the covariance matrix is the only scatter functional that is additive and the only scatter functional that is full affine equivariant.
\end{abstract}

\begin{keyword}
Additivity\sep affine equivariance\sep independence\sep scatter functional
\MSC[2010] 62H05 \sep 62H20
\end{keyword}

\end{frontmatter}

%\linenumbers

\section{The properties of the covariance matrix}

Let $\left( \Omega, \mathcal{F}, \mathbb{P} \right)$ be a probability space and let $\mathcal{D}(\mathbb{R})$ be a suitable, rich enough class of square integrable random variables ${x}: \Omega \rightarrow \mathbb{R}$ that contains the standard normal distribution. Assume further that the product spaces $ \mathcal{D}(\mathbb{R}^p) = \mathcal{D}(\mathbb{R}) \times \cdots \times \mathcal{D}(\mathbb{R}) $ are closed under affine transformations: $ \textbf{A} \textbf{x} + \textbf{b} \in \mathcal{D}(\mathbb{R}^p)$ for all $ p \in \mathbb{N} $, $ \textbf{x} = (x_1, \ldots , x_p)^\top \in \mathcal{D}(\mathbb{R}^p)$, invertible $\textbf{A} \in \mathbb{R}^{p \times p}$ and $\textbf{b} \in \mathbb{R}^p$. Fixing next $ p $, the ordinary covariance matrix $\mathrm{Cov}:\mathcal{D}(\mathbb{R}^p) \rightarrow \mathbb{R}^{p \times p}_{++}$ can now be thought of as a functional from the class $\mathcal{D}(\mathbb{R}^p)$ to the space $\mathbb{R}^{p \times p}_{++}$ of all $p \times p$ positive-definite matrices and any standard text in probability lists its following properties:

\begin{itemize}
	\item[\#1] $\mathrm{Cov}$ is affine equivariant in the sense that $\mathrm{Cov}(\textbf{A} \textbf{x} + \textbf{b}) = \textbf{A} \mathrm{Cov}(\textbf{x}) \textbf{A}^\top$ for all $\textbf{x} \in \mathcal{D}(\mathbb{R}^p)$, invertible $\textbf{A} \in \mathbb{R}^{p \times p}$ and $\textbf{b} \in \mathbb{R}^p$.
	\item[\#2] $\mathrm{Cov}$ is additive in the sense that $\mathrm{Cov}(\textbf{x} + \textbf{y}) = \mathrm{Cov}(\textbf{x}) + \mathrm{Cov}(\textbf{y})$ for all pairs $\textbf{x}, \textbf{y} \in \mathcal{D}(\mathbb{R}^p)$ such that $ \textbf{x} $ and $ \textbf{y} $ are independent.
	\item[\#3] $\mathrm{Cov}$ has the independence property, i.e., $\{ \mathrm{Cov}( \textbf{x} ) \}_{jk} = 0$ for all $\textbf{x} \in \mathcal{D}(\mathbb{R}^p)$ with independent $j$th and $k$th components. 
\end{itemize}

%In listing the properties we have used the standard abuse of notation, $\mathrm{Cov}(\textbf{x}) \equiv \mathrm{Cov}(F_\textbf{x})$, where $F_\textbf{x}$ denotes the distribution of $\textbf{x}$.
As is common in mathematics, isolating the key characteristics of an object and studying their implications with no reference to the original object yields a rich theory also in the context of the previous three properties. If we focus on the first of them, affine equivariance, and consider the class of all positive-definite matrix-valued functionals which satisfy $\#1$, we obtain what are called \textit{scatter functionals}, studied extensively especially in the robust community as alternative measures of dispersion to the covariance matrix. Famous examples of scatter functionals include e.g. the M-functionals, S-functionals and the minimum covariance determinant (MCD) estimator, see \cite{tyler2009invariant, rousseeuw2013high, dumbgen2015m} and the references therein.

\begin{definition}
	A functional $\textbf{S}:\mathcal{D}(\mathbb{R}^p) \rightarrow \mathbb{R}^{p \times p}_{++}$ is a scatter functional if $\textbf{S}(\textbf{A} \textbf{x} + \textbf{b}) = \textbf{A} \textbf{S}(\textbf{x}) \textbf{A}^\top$ for all $\textbf{x} \in \mathcal{D}(\mathbb{R}^p)$, invertible $\textbf{A} \in \mathbb{R}^{p \times p}$ and $\textbf{b} \in \mathbb{R}^p$.
\end{definition}

We will denote the class of all scatter functionals by $\mathcal{S}$ and as all our following work will happen in $\mathcal{S}$, the location invariance in $ \#1 $ guarantees that we may without loss of generality assume $\mathrm{E}(\textbf{x}) = \textbf{0}$. Furthermore, as the structural properties of $\textbf{S} \in \mathcal{S}$ and $c \textbf{S}$ are the same for all scalars $c > 0$, we will ``standardize'' the members of $\mathcal{S}$ in order to obtain a single representative from each of the equivalence classes. Recall that if $\textbf{z} \in \mathcal{D}(\mathbb{R}^p)$ has a spherical distribution \cite{fang1990symmetric}, that is, $\textbf{z} \sim \textbf{U} \textbf{z}$ for all orthogonal matrices $\textbf{U} \in \mathbb{R}^{p \times p}$, then all scatter functionals evaluated at $\textbf{z}$ are proportional to the identity matrix. Thus any spherical distribution can be used as a reference in the standardization and we choose to use the most common one, the multivariate standard normal distribution, and restrict $\mathcal{S}$ to contain precisely those scatter functionals for which $\textbf{S}(\textbf{z}) = \textbf{I}_p$ for $\textbf{z} \sim  \mathcal{N}(\textbf{0}, \textbf{I}_p)$. As a consequence, if $\textbf{x} \sim  \mathcal{N}(\textbf{0}, \boldsymbol{\Sigma})$ for some $\boldsymbol{\Sigma}$ then $\textbf{S}(\textbf{x}) = \boldsymbol{\Sigma}$ for any $\textbf{S} \in \mathcal{S}$.

The properties $\#2$ and $\#3$ are much less studied than the affine equivariance, and never in isolation of $\#1$. \cite{nordhausen2015cautionary} express their interest of finding out whether any other functional than $\mathrm{Cov}$ simultaneously satisfies $\#1$ and $\#2$, citing factor analysis and structural equation modelling as motivations: in any multivariate model with two independent sources of variation, $\textbf{x} = \textbf{z} + \boldsymbol{\epsilon}$, $\textbf{z}, \boldsymbol{\epsilon} \in \mathcal{D}(\mathbb{R}^p)$, the ``variance decomposition'' $\textbf{S}(\textbf{x}) = \textbf{S}(\textbf{z}) + \textbf{S}(\boldsymbol{\epsilon})$ holds only if $\textbf{S} \in \mathcal{S}$ satisfies $\#2$ in addition to $\#1$. Thus the plugging-in of an arbitrary scatter functional without the additivity property to replace the covariance matrix is unwarranted if the theory is dependent on the previous decomposition, such as in the classical factor analysis. In Section \ref{sec:results} we will prove that, under natural regularity assumptions, $\mathrm{Cov}$ is actually the only additive scatter functional in the sense of the next definition, where we have implicitly assumed that the space $ \mathcal{D}(\mathbb{R}^p) $ is closed under addition.

\begin{definition}\label{def:additive_scatter}
	The scatter functional $\textbf{S}:\mathcal{D}(\mathbb{R}^p) \rightarrow \mathbb{R}^{p \times p}_{++}$ is additive if $\textbf{S}(\textbf{x} + \textbf{y}) = \textbf{S}(\textbf{x}) + \textbf{S}(\textbf{y})$ for all pairs $\textbf{x}, \textbf{y} \in \mathcal{D}(\mathbb{R}^p)$ such that $ \textbf{x} $ and $ \textbf{y} $ are independent.
\end{definition}

Simple induction now reveals that any additive scatter functional satisfies $\textbf{S}(\sum_{i=1}^{n} \textbf{x}_i) = \sum_{i=1}^n \textbf{S}(\textbf{x}_i)$ for any finite number $ n $ of independent random vectors $ \textbf{x}_i \in \mathcal{D}(\mathbb{R}^p)$, $ i = 1, \ldots , n $.

The property $\#3$, and especially its certain variant, is more explored in the literature and we will return to them in the discussion section. Instead, we will discuss a stronger form of $\#1$ which also holds for $\mathrm{Cov}$.

\begin{itemize}
\item[\#4] $\mathrm{Cov}$ is affine equivariant in the sense that $\mathrm{Cov}(\textbf{A} \textbf{x} + \textbf{b}) = \textbf{A} \mathrm{Cov}(\textbf{x}) \textbf{A}^\top$ for all $ k \in \mathbb{N} $, $\textbf{x} \in \mathcal{D}(\mathbb{R}^p)$, $\textbf{A} \in \mathbb{R}^{k \times p}$ and $\textbf{b} \in \mathbb{R}^k$.
\end{itemize}
We call the above property \textit{full affine equivariance} and note that it relies on abuse of notation in the sense that it uses $\mathrm{Cov}$ both as a function on $\mathcal{D}(\mathbb{R}^p)$ and as a function on $\mathcal{D}(\mathbb{R}^k)$. Thus, to be completely rigorous, we provide a definition for full affine equivariant scatter functionals which takes into account the multiple dimensionalities.

\begin{definition}\label{def:fae_scatter}
	The scatter functional $\textbf{S}:\mathcal{D}(\mathbb{R}^p) \rightarrow \mathbb{R}^{p \times p}_{+}$ is full affine equivariant if there exists a countable set of functionals $\mathbb{S} = \{ \textbf{S}_k:\mathcal{D}(\mathbb{R}^k) \rightarrow \mathbb{R}^{k \times k}_{+} \mid  k \in \mathbb{N} \}$ satisfying
	\begin{itemize}
		\item[i)] $ \textbf{S}(\textbf{x}) = \textbf{S}_p(\textbf{x}) $ for all $\textbf{x} \in \mathcal{D}(\mathbb{R}^p)$,
		\item[ii)] $\textbf{S}_k(\textbf{A} \textbf{x} + \textbf{b}) = \textbf{A} \textbf{S}_\ell(\textbf{x}) \textbf{A}^\top$ for all $ k, \ell \in \mathbb{N} $, $ \textbf{x} \in \mathcal{D}(\mathbb{R}^\ell) $ $\textbf{A} \in \mathbb{R}^{k \times \ell}$ and $\textbf{b} \in \mathbb{R}^k$.
	\end{itemize}
\end{definition}

Implicit in Definition \ref{def:fae_scatter} is the assumption that the countable sequence of spaces $ \mathcal{D}(\mathbb{R}), \mathcal{D}(\mathbb{R}^2), \ldots $ is closed with respect to affine transformations in the sense that $\textbf{A} \textbf{x} + \textbf{b} \in \mathcal{D}(\mathbb{R}^k)$ for all $ k, \ell \in \mathbb{N} $, $ \textbf{x} \in \mathcal{D}(\mathbb{R}^\ell) $ $\textbf{A} \in \mathbb{R}^{k \times \ell}$ and $\textbf{b} \in \mathbb{R}^k$. The inclusion of rank-deficient affine transformations has also caused us to broaden the definition of scatter functionals to be able to take values that are singular, hence the change of range from positive-definite to positive semidefinite matrices $ \mathbb{R}^{p \times p}_{+} $ in Definition \ref{def:fae_scatter}.

In the light of Definition \ref{def:fae_scatter} it is obvious that the property $\#4$ makes in some sense a much stronger statement about the parent scatter functional $\textbf{S} \in \mathcal{S}$ than $\#1$, $\#2$ and $\#3$ in that it makes a connection between the instances of $\textbf{S}$ in different dimensions. In fact, full affine equivariance turns out to be such a strong property that, under the same regularity conditions that we exercise with additivity, again $ \mathrm{Cov} $ is the only scatter functional enjoying it.

\section{Two regularity conditions}

Before moving on to our main results we introduce and comment on the regularity conditions we will impose on all of our scatter functionals in the following. The first of these is continuity with respect to the convergence in distribution to the normal distribution.

\begin{definition}\label{def:normal_continuous}
	A scatter functional $\textbf{S} \in \mathcal{S}$ is normal continuous if for any sequence of zero-mean random variables $\textbf{x}_n \in \mathcal{D}(\mathbb{R}^p)$ converging in distribution to $\mathcal{N}(\textbf{0}, \boldsymbol{\Sigma})$ for $ \boldsymbol{\Sigma} \in \mathbb{R}^{p \times p}_{++} $, we have $\textbf{S}(\textbf{x}_n) \rightarrow \boldsymbol{\Sigma}$ in some matrix norm $\| \cdot \|$.
\end{definition}

The normal continuity is a special instance of weak continuity and sufficient conditions for the latter in the case of M-functionals of scatter with known location are given in Section 6.1 of \cite{dumbgen2015m}. In particular, any M-functional of scatter with bounded and strictly increasing $ \psi $-function is weakly continuous. Many popular robust scatter functionals such as the multivariate $ t $-functionals of scatter \citep{dumbgen2005breakdown} fall into this class.

Our second condition says that the value of a scatter functional must be uniquely defined by the marginal distribution of its argument and not be dependent on any external variables.

\begin{definition}\label{def:self_contained}
	A scatter functional $ \textbf{S} \in \mathcal{S} $ is self-contained if $ \textbf{S}(\textbf{x}_1) = \textbf{S}(\textbf{x}_2) $ for all pairs $ \textbf{x}_1, \textbf{x}_2 \in \mathcal{D}(\mathbb{R}^p) $ such that $ \textbf{x}_1 $ and $ \textbf{x}_2 $ are identically distributed.
\end{definition}

The assumption on self-containedness is needed to exclude scatter functionals such as $\textbf{S}_{SIR} (\textbf{x}) = \mathrm{E}\{ \mathrm{E}(\textbf{x} \mid y) \mathrm{E}(\textbf{x} \mid y)^\top  \} $ where $ y $ is some suitable ``response'' variable. $ \textbf{S}_{SIR} $, which is used in sliced inverse regression \cite{li1991sliced}, is clearly full affine equivariant but has the same structure with respect to the argument $ \textbf{x} $ as the covariance matrix and is as such fundamentally equivalent to $ \mathrm{Cov} $.

\section{The main results}\label{sec:results}

We next state and prove our main results that under the regularity conditions of the previous section $ \mathrm{Cov} $ is the only additive scatter functional and the only full affine equivariant scatter functional. Various lemmas that are needed for the proof of Proposition \ref{prop:fae} below are collected in \ref{sec:appendix}.

\begin{proposition}\label{prop:additivity}
Let $\textbf{S} \in \mathcal{S}$ be normal continuous, self-contained and additive. Then $\textbf{S} = \mathrm{Cov}$. 
\end{proposition}

\begin{proof}
Let $\textbf{x} \in \mathcal{D}(\mathbb{R}^p)$ be arbitrary with zero mean and covariance matrix $\mathrm{Cov}(\textbf{x}) = \boldsymbol{\Sigma}$ and let $\textbf{x}_1, \ldots , \textbf{x}_n$ be a random sample from the distribution of $\textbf{x}$. Then by additivity and affine equivariance,
\[
\textbf{S} \left( \frac{1}{\sqrt{n}} \sum_{i = 1}^n \textbf{x}_i \right) = \frac{1}{n} \textbf{S}(\textbf{x}_1) + \ldots + \frac{1}{n} \textbf{S}(\textbf{x}_n) = \textbf{S}(\textbf{x}),
\]
where the self-containedness of $ \textbf{S} $ is used to obtain the final equality. Now, by the central limit theorem $(1/\sqrt{n}) \sum_{i = 1}^n \textbf{x}_i \rightarrow_D \mathcal{N}(\textbf{0}, \boldsymbol{\Sigma})$ and the continuity of $\textbf{S}$ means that for every $\epsilon > 0$ there exists $N = N(\epsilon)$ such that
\[
\| \textbf{S} \left( \frac{1}{\sqrt{n}} \sum_{i = 1}^n \textbf{x}_i \right) - \textbf{S} \left\{ \mathcal{N}(\textbf{0}, \boldsymbol{\Sigma}) \right\} \| < \epsilon, \quad \forall n > N.
\]
But now $\textbf{S} \{ (1/\sqrt{n}) \sum_{i = 1}^n \textbf{x}_i \} = \textbf{S}(\textbf{x}) $ and $\textbf{S} \left\{ \mathcal{N}(\textbf{0}, \boldsymbol{\Sigma}) \right\} = \boldsymbol{\Sigma}$, meaning that
\[
\forall \epsilon > 0, \quad \| \textbf{S} \left( \textbf{x} \right) - \boldsymbol{\Sigma} \| < \epsilon,
\]
and consequently, $\textbf{S}(\textbf{x}) = \boldsymbol{\Sigma} = \mathrm{Cov}(\textbf{x})$.
\end{proof}

\begin{proposition}\label{prop:fae}
	Let $\textbf{S} \in \mathcal{S}$ be normal continuous, self-contained and full affine equivariant. Then $\textbf{S} = \mathrm{Cov}$. 
\end{proposition}

\begin{proof}
By the definition of full affine equivariant scatter functionals there exists a countable set of scatter functionals $ \textbf{S}_1, \textbf{S}_2 \ldots $ that satisfies the conditions of Definition \ref{def:fae_scatter}.
 Letting now $ \textbf{x} = (x_1, \ldots , x_p) \in \mathcal{D}(\mathbb{R}^p) $ be arbitrary with zero mean and covariance matrix $\mathrm{Cov}(\textbf{x}) = \boldsymbol{\Sigma}$, by Lemma \ref{lem:subvector} the first diagonal element of $ \textbf{S}(\textbf{x}) $ is $\{\textbf{S}(\textbf{x})\}_{11} = \textbf{S}_1(x_1)$. Take now a random sample $ \textbf{y} = ( y_1, \ldots , y_n )^\top $ from the distribution of $ x_1 $. Then by full affine equivariance and Lemma \ref{lem:subvector},
 \begin{equation}\label{eq:fae_expansion}
 \textbf{S}_1 \left( \frac{1}{\sqrt{n}} \sum_{i = 1}^n y_i \right) = \frac{1}{n} \mathds{1}^\top \textbf{S}_n(\textbf{y}) \mathds{1} = \frac{1}{n} \sum_{i=1}^n \textbf{S}_1(y_i) + \frac{1}{n}  \sum_{i=1}^n  \sum_{j = 1, j \neq i}^n \{ \textbf{S}_2(\textbf{y}_{ij}) \}_{1 2},
 \end{equation}
 where $\textbf{y}_{ij} = (y_i, y_j)^\top$. By Lemma \ref{lem:self_contained} both $ \textbf{S}_1 $ and $ \textbf{S}_2 $ are self-contained (by Lemma \ref{lem:self_contained_2} in the case $ p = 1 $) and thus $ \textbf{S}_1(y_i) $ equals $\textbf{S}_1(x_1)$ for all $ i = 1, \ldots , n $ and $ \{ \textbf{S}_2(\textbf{y}_{ij}) \}_{1 2}$ is equal to some $ c(x_1) \in \mathbb{R} $, independent of $ i, j, n $.
 
 Thus \eqref{eq:fae_expansion} gets the form,
 \begin{equation}\label{eq:fae_expansion_2} 
 \textbf{S}_1 \left( \frac{1}{\sqrt{n}} \sum_{i = 1}^n y_i \right) = \textbf{S}_1(x_1) + (n - 1)  c(x_1).
 \end{equation}
 Assume next that $ c(x_1) \neq 0 $ and take the limit with respect to $ n $ on both sides of \eqref{eq:fae_expansion_2}. By Lemma \ref{lem:normal_continuous}, $ \textbf{S}_1 $ is normal continuous and consequently the left-hand side of \eqref{eq:fae_expansion_2} approaches $ \mathrm{Var}(y_1) = \mathrm{Var}(x_1) = (\boldsymbol{\Sigma})_{11} $, whereas the right-hand side approaches either $ \infty $ or $ -\infty $, yielding a contradiction. Thus we must have $ c(x_1) = 0 $ and by the same limiting argument we obtain $ \textbf{S}_1(x_1) = (\boldsymbol{\Sigma})_{11} $. As the choice of the element $ x_1 $ was arbitrary we have that the diagonal of $ \textbf{S}(\textbf{x}) $ consists of the variances of the components of $ \textbf{x} $. If $ p = 1 $ we are now done, so assume in the following that $ p > 1 $.

We next show that the $ (1, 2) $ off-diagonal element of $ \textbf{S}(\textbf{x}) $ is equal to $ (\boldsymbol{\Sigma})_{12} $, the covariance between $ x_1 $ and $ x_2 $. For that, consider the random vector $ \textbf{y} = (y_1 + y_2, z^*_2, \ldots , z_p^* ) $ where $ (y_1, y_2)^\top $ has the same joint distribution as $ (x_1, x_2)^\top $ and $ z^*_2, \ldots , z_p^* $ are independent standard normal random variables. We next write the first diagonal element of $ \textbf{S}_p(\textbf{y}) $ in two different ways. First, based on the previous paragraph, we have 
\begin{equation}\label{eq:sum_expansion_1}
\{\textbf{S}_p(\textbf{y})\}_{11} = \mathrm{Var}(y_1 + y_2) = \mathrm{Var}(x_1 + x_2) = (\boldsymbol{\Sigma})_{11} + 2 (\boldsymbol{\Sigma})_{12} + (\boldsymbol{\Sigma})_{22}.
 \end{equation}
Second, by Lemma \ref{lem:subvector} and full affine equivariance we have
\begin{equation}\label{eq:sum_expansion_2}
\{\textbf{S}_p(\textbf{y})\}_{11} = \textbf{S}_1(y_1 + y_2) = \mathds{1}^\top \textbf{S}_2\{(y_1, y_2)^\top) \} \mathds{1}.
 \end{equation}
Using Lemmas \ref{lem:subvector}, \ref{lem:self_contained} and the equality of the diagonals of $ \textbf{S}(\textbf{x}) $ and  $ \boldsymbol{\Sigma} $ we obtain that the three unique elements of $ \textbf{S}_2\{(y_1, y_2)^\top) \} $ are equal to $ \textbf{S}_1(y_1) = \textbf{S}_1(x_1) = (\boldsymbol{\Sigma})_{11}$, $ \textbf{S}_1(y_2) = \textbf{S}_1(x_2) = (\boldsymbol{\Sigma})_{22}$ and $ \textbf{S}_2\{(y_1, y_2)^\top) \}_{12} = \textbf{S}_2\{(x_1, x_2)^\top) \}_{12} = \{\textbf{S}(\textbf{x})\}_{12} $. Plugging these in to \eqref{eq:sum_expansion_2} and equating with \eqref{eq:sum_expansion_1} now yields the desired equality,  $ \{\textbf{S}(\textbf{x})\}_{12} = (\boldsymbol{\Sigma})_{12} $. As the choice of the elements $ x_1 $ and $ x_2 $ was arbitrary, the analogous result holds for all off-diagonal elements of $ \textbf{S}(\textbf{x}) $, and consequently, $\textbf{S}(\textbf{x}) = \boldsymbol{\Sigma} = \mathrm{Cov}(\textbf{x})$.

\end{proof}

\section{Discussion}

In this short note we proved that under natural regularity conditions covariance matrix is the only scatter functional satisfying $ \#2 $ (additivity) or $ \#4 $ (full affine equivariance). The main technique used in the proofs of Propositions \ref{prop:additivity} and \ref{prop:fae} was quite elementary, requiring passing a limit to the argument of the scatter functional and invoking the central limit theorem to establish a connection to the normal distribution. The same approach, however, seems not to be sufficient to prove or disprove the conjecture of \cite{nordhausen2015cautionary} that also $ \# 3$ is a characterizing property of $ \mathrm{Cov} $.  

A weaker version of $ \#3 $, the \textit{joint independence property}, has received more attention and is not unique to the covariance matrix.  A scatter functional $ \textbf{S} \in \mathcal{S} $ has the joint independence property if $ \textbf{S}(\textbf{x}) $ is diagonal whenever the components of $ \textbf{x} $ are independent. For example, $ \textbf{S}(\textbf{x}) = \mathrm{E} \{ \textbf{x} \textbf{x}^\top \mathrm{Cov}(\textbf{x})^{-1} \textbf{x} \textbf{x}^\top \} $ possesses the property and so does any symmetrized scatter functional, $ \textbf{S}^*:\mathcal{D}(\mathbb{R}^p) \rightarrow \mathbb{R}^{p \times p}_{++} $ with $ \textbf{x} \mapsto \textbf{S}(\textbf{x}_1 - \textbf{x}_2) $ where $ \textbf{S} \in \mathcal{S} $ and $ \textbf{x}_1, \textbf{x}_2 $ are independent copies of $ \textbf{x} $, see \cite{oja2006scatter, taskinen2007independent} where pairs of scatter functionals with the joint independence property are used to solve the independent component problem. An open question regarding the joint independence property, remarked also in \cite{nordhausen2015cautionary}, is whether the set of all symmetrized scatter functionals is actually equal to the set of all scatter functionals with the joint independence property.

A further interesting question is whether our results generalize also in some form to \textit{shape functionals}. A functional $\textbf{S}:\mathcal{D}(\mathbb{R}^p) \rightarrow \mathbb{R}^{p \times p}_{++}$ is called a shape functional if it is affine equivariant up to proportionality: $\textbf{S}(\textbf{A} \textbf{x} + \textbf{b}) = c \textbf{A} \textbf{S}(\textbf{x}) \textbf{A}^\top$ for all $\textbf{x} \in \mathcal{D}(\mathbb{R}^p)$, invertible $\textbf{A} \in \mathbb{R}^{p \times p}$, $\textbf{b} \in \mathbb{R}^p$ and some $ c = c(\textbf{S}, \textbf{x}, \textbf{A}) \in \mathbb{R}$. Some inspection reveals that at least the limit trick used in this note can not be used in the context of shape functionals as the constant of proportionality $ c $ may also depend on the sample size $ n $ after having pulled out the factor $ 1/\sqrt{n} $.

%Throughout the paper we assumed that property $\#1$ holds and explored the implications of pairing it with the other properties. A complementary approach would be to take either $ \#2 $ or $ \#3 $ as the defining property of a class of functionals and study the properties of e.g. all additive functionals $\textbf{S}:\mathcal{D}(\mathbb{R}^p) \rightarrow \mathbb{R}^{p \times p}_{++}$. Although perhaps not of much interest to the practical statistician as property $\#1$ is what makes the functionals actually measure dispersion, the question is still mathematically valid and interesting. 

\appendix

\section{Auxiliary results}\label{sec:appendix}

In the following four lemmas we assume that $ \textbf{S} \in \mathcal{S} $ is a fixed, arbitrary full affine equivariant scatter functional and $ \mathbb{S} $ is the countable set of functionals
related to it described in Definition \ref{def:fae_scatter}.

The first lemma says that any of the functionals $ \textbf{S}_k $, $ k \in \mathbb{N} $, can be built up from the ``variances'' and ``covariances'' found in $ \textbf{S}_1 $ and $ \textbf{S}_2 $.

\begin{lemma}\label{lem:subvector}
	Let $ k \in \mathbb{N} $ and $\textbf{x} = (x_1, \ldots , x_k) \in \mathcal{D}(\mathbb{R}^k)$. Then
	\begin{align*}
	\left\{ \textbf{S}_k ( \textbf{x} ) \right\}_{ii} &= \textbf{S}_1( \textbf{x}_{i} ) \\
	\left\{ \textbf{S}_k ( \textbf{x} ) \right\}_{ij} &= \left\{ \textbf{S}_2 ( \textbf{x}_{ij} ) \right\}_{12},
	\end{align*}
	for all $ i \neq j = 1, \ldots , k$ where $\textbf{x}_{ij} = (x_i, x_j)^\top$. 
\end{lemma}
\begin{proof}
	Let $ \textbf{e}_{ik} \in \mathbb{R}^k$ refer to the $ i $th standard basis vector of $ \mathbb{R}^k $. The first claim follows then by observing that
	\[ 
	\textbf{S}_1 ( \textbf{x}_{i} ) = \textbf{S}_1 ( \textbf{e}_{ik}^\top \textbf{x} ) = \textbf{e}_{ik}^\top \textbf{S}_k (\textbf{x} ) \textbf{e}_{ik} =  \left\{ \textbf{S}_k ( \textbf{x} ) \right\}_{ii},
	\]
	and the second claim comes similarly from
	\[ 
	\left\{ \textbf{S}_2 ( \textbf{x}_{ij} ) \right\}_{12} = \textbf{e}_{12}^\top \textbf{S}_2 ( \textbf{x}_{ij} ) \textbf{e}_{22} = \textbf{e}_{12}^\top \textbf{S}_2 \{ (\textbf{e}_{ik}, \textbf{e}_{jk})^\top \textbf{x} \} \textbf{e}_{22} = \textbf{e}_{ik}^\top \textbf{S}_k ( \textbf{x} ) \textbf{e}_{jk} = \left\{ \textbf{S}_k ( \textbf{x} ) \right\}_{ij},
	\]
	where the second-to-last equalities use the full affine equivariance.
\end{proof}

The next two lemmas state that if $ \textbf{S} $ is either normal continuous or self-contained the respective property is carried over also to $ \textbf{S}_1, \ldots \textbf{S}_{p} $ and the final one says that in the case of $ p = 1 $ we can also move one step to the opposite direction and the self-containedness of $ \textbf{S} $ implies the self-containedness of $ \textbf{S}_2 $. 

\begin{lemma}\label{lem:normal_continuous}
	Assume that $ \textbf{S}$ is normal continuous. Then each of the functionals $ \textbf{S}_k $, $ k \in \mathbb{N}, k \leq p $, is also normal continuous.
\end{lemma}

\begin{proof}
	The case $ \textbf{S}_p = \textbf{S} $ follows trivially so fix $ k < p $ and let $ \textbf{x}_n \in \mathcal{D}(\mathbb{R}^k) $ be a sequence of random variables converging in distribution to $\mathcal{N}(\textbf{0}, \boldsymbol{\Sigma})$ for some $ \boldsymbol{\Sigma} \in \mathbb{R}^{k \times k}_{++} $. Construct then a sequence $ \textbf{x}^*_n = (\textbf{x}_n, x^*_{k+1}, \ldots, x^*_p)^\top $ of $ p $-variate random vectors where $ x^*_{k+1}, \ldots, x^*_p $ are independent standard normal random variables independent of the sequence $ \textbf{x}_n $. Then the sequence $ \textbf{x}^*_n $ converges in distribution to $\mathcal{N}\{ \textbf{0}, \mathrm{diag}( \boldsymbol{\Sigma}, \textbf{I}_{p-k}) \}$ and by the normal continuity of $ \textbf{S}_p = \textbf{S} $ we have that $	\textbf{S}_p(\textbf{x}^*_n) \rightarrow \mathrm{diag}( \boldsymbol{\Sigma}, \textbf{I}_{p-k})$. Letting $ \textbf{E}_k $ contain the first $ k $ columns of the $ p \times p$ identity matrix, we have by the continuous mapping theorem,
	\[ 
	\textbf{S}_k(\textbf{x}_n) = \textbf{S}_k(\textbf{E}_k^\top \textbf{x}^*_n) = \textbf{E}_k^\top \textbf{S}_p(\textbf{x}^*_n) \textbf{E}_k \rightarrow \textbf{E}_k^\top \mathrm{diag}( \boldsymbol{\Sigma}, \textbf{I}_{p-k}) \textbf{E}_k = \boldsymbol{\Sigma},
	\]
	and hence $ \textbf{S}_k $ is normal continuous.
\end{proof}

\begin{lemma}\label{lem:self_contained}
	Assume that $ \textbf{S}$ is self-contained. Then each of the functionals $ \textbf{S}_k $, $ k \in \mathbb{N}, k \leq p $, is also self-contained.
\end{lemma}

\begin{proof}
	The case $ \textbf{S}_p = \textbf{S} $ follows trivially so fix $ k < p $ and let $ \textbf{x}, \textbf{y} \in \mathcal{D}(\mathbb{R}^k) $ be identically distributed random vectors. As in the proof of Lemma \ref{lem:normal_continuous} we fill up $ \textbf{x}, \textbf{y} \in \mathcal{D}(\mathbb{R}^p)$ with independent standard normal random variables to obtain the identically distributed $ \textbf{x}^*, \textbf{y}^* $. By the self-containedness of $ \textbf{S}_p = \textbf{S} $ we have $ \textbf{S}_p(\textbf{x}^*) = \textbf{S}_p(\textbf{y}^*)$ and letting $ \textbf{E}_k $ be as in the proof of Lemma \ref{lem:normal_continuous}, we obtain,
	\[ 
	\textbf{S}_k(\textbf{x}) = \textbf{S}_k(\textbf{E}_k^\top \textbf{x}^*) = \textbf{E}_k^\top \textbf{S}_p(\textbf{x}^*) \textbf{E}_k = \textbf{E}_k^\top \textbf{S}_p(\textbf{y}^*) \textbf{E}_k = \textbf{S}_k(\textbf{E}_k^\top \textbf{y}^*) = \textbf{S}_k(\textbf{y}),
	\]
	concluding the proof.
\end{proof}

\begin{lemma}\label{lem:self_contained_2}
	Let $ p = 1 $ and assume that $ \textbf{S}$ is self-contained. Then $ \textbf{S}_2 $ is also self-contained.
\end{lemma}

\begin{proof}
	Let $ \textbf{x} = (x_1, x_2), \textbf{y} = (y_1, y_2) \in \mathcal{D}(\mathbb{R}^2) $ be identically distributed random vectors. By Lemma \ref{lem:subvector} and the self-containedness of $ \textbf{S}_1 $, the matching diagonal elements of $ \textbf{S}_2(\textbf{x}) $ and $ \textbf{S}_2(\textbf{y}) $ are equal and for our claim it is sufficient to show that $ \{\textbf{S}_2(\textbf{x})\}_{12} = \{\textbf{S}_2(\textbf{y})\}_{12} $. Again by the self-containedness of $ \textbf{S}_1 $, we have $ \textbf{S}_1(x_1 + x_2) = \textbf{S}_1(y_1 + y_2) $ and by full affine equivariance and Lemma \ref{lem:subvector} this is equivalent to
	\[
	\textbf{S}_1(x_1) + 2 \{\textbf{S}_2(\textbf{x})\}_{12} + \textbf{S}_1(x_2) = \textbf{S}_1(y_1) + 2 \{\textbf{S}_2(\textbf{y})\}_{12} + \textbf{S}_1(y_2).
	\]
	Applying once more the self-containedness of $ \textbf{S}_1 $ now yields the desired result.
\end{proof}

%\section*{References}

\bibliography{mybibfile}

\end{document}